\documentclass[11pt]{amsart}
\usepackage{amssymb,amsmath,amsthm}
\usepackage{enumerate, url, verbatim}
\usepackage{comment}
\newcommand{\leg}[2]{\genfrac{(}{)}{}{}{#1}{#2}}

\newtheorem{theorem}{Theorem}
\newtheorem{lemma}[theorem]{Lemma}

\newtheorem{proposition}[theorem]{Proposition}

\theoremstyle{remark}
\newtheorem{definition}[theorem]{Definition}
\newtheorem{remarks}[theorem]{Remarks}

\newtheorem*{remark}{Remark}

\numberwithin{theorem}{section} \numberwithin{equation}{section}

\newcommand{\al}{\alpha}

\newcommand{\mfa}{\mathfrak{a}}

\newcommand{\ov}[1]{\overline{#1}}

\newcommand{\calD}{\mathcal{D}}
\newcommand{\calP}{\mathcal{P}}

\newcommand{\FF}{\mathcal{F}}

\newcommand{\om}{\omega}
\newcommand{\mfp}{\mathfrak{p}}
\newcommand{\mfb}{\mathfrak{b}}
\newcommand{\mfq}{\mathfrak{q}}
\newcommand{\mff}{\mathfrak{f}}
\newcommand{\mfc}{\mathfrak{c}}

\newcommand{\calL}{\mathcal{L}}

\newcommand{\calB}{\mathcal{B}}
\newcommand{\calF}{\mathcal{F}}

\newcommand{\F}{\mathbb{F}}
\newcommand{\Cl}{{\text {\rm Cl}}}

\newcommand{\rk}{{\text {\rm rk}}}

\newcommand{\Q}{\mathbb{Q}}

\newcommand{\Z}{\mathbb{Z}}

\newcommand{\LL}{{\mathcal L}}
\newcommand{\p}{\mathfrak p}

\newcommand{\q}{\mathfrak q}
\newcommand{\wK}{\widetilde{K}}
\renewcommand{\c}{\mathfrak c}

\newcommand{\N}{{\mathcal N}}

\newcommand{\textmod}{{\text {\rm mod}}}

\newcommand{\Gal}{{\text {\rm Gal}}}

\newcommand{\Disc}{\textnormal{Disc}}

\newcommand{\Ker}{\textnormal{Ker}}

\renewcommand{\b}{{\mathfrak b}}

\begin{document}
\title[Dirichlet Series Associated to Cubic Fields with Given Quadratic Resolvent]
{Dirichlet Series Associated to Cubic Fields with Given Quadratic Resolvent}

\author{Henri Cohen}
\address{Universit\'e Bordeaux I, Institut de Math\'ematiques, U.M.R. 5251 du C.N.R.S,
351 Cours de la Lib\'eration,
33405 TALENCE Cedex, FRANCE}
\email{Henri.Cohen@math.u-bordeaux1.fr}
\author{Frank Thorne}
\address{Department of Mathematics, University of South Carolina, 1523 Greene Street, Columbia, SC 29208, USA}
\email{thorne@math.sc.edu}

\begin{abstract}

Let $k$ be a quadratic field. We give an explicit formula for the Dirichlet
series $\sum_K|\Disc(K)|^{-s}$, where the sum is over isomorphism classes of
all cubic fields whose quadratic resolvent field is isomorphic to $k$.

Our work is a sequel to \cite{CM} (see also \cite{M}), where such formulas 
are proved in a more general setting, in terms of sums over characters of 
certain groups related to ray class groups. In the present paper we carry the
analysis further and prove explicit formulas for these Dirichlet series over 
$\Q$, and in a companion paper we do the same for quartic fields having a given
cubic resolvent.

As an application, we compute tables of the number of $S_3$-sextic
fields $E$ with $|\Disc(E)| < X$, for $X$ ranging up to $10^{23}$.
An accompanying PARI/GP implementation is available from the second author's website.

\end{abstract}

\maketitle
\section{Introduction}
A classical problem in algebraic number theory is that of
{\itshape enumerating number fields} by discriminant. Let $N^{\pm}_d(X)$ 
denote the number of isomorphism classes of number fields $K$ with 
$\deg(K) = d$ and $0 < \pm \Disc(K) < X$.
The quantity $N^{\pm}_d(X)$ has seen a great deal of study; see (for example) 
\cite{CDO, B_icm, T_four} for surveys of classical and more recent work.

It is widely believed that $N^{\pm}_d(X) = C^{\pm}_d X + o(X)$ for all 
$d\ge2$. For $d = 2$ this is classical, and the case $d = 3$ was proved in 
1971 work of Davenport and Heilbronn \cite{DH}. The cases $d = 4$ and $d = 5$
were proved much more recently by Bhargava \cite{B4, B5}. In addition, 
Bhargava in \cite{B_conj} also conjectured a value of the constants 
$C^{\pm}_{d,S_d}$ for $d > 5$, where the additional index $S_d$ means that 
one counts only degree $d$ number fields with Galois group of the Galois 
closure isomorphic to $S_d$.

Related questions have also seen recent attention. For example, Belabas 
\cite{Bel} developed and implemented a fast algorithm to compute large tables
of cubic fields, which has proved essential for subsequent numerical 
computations (including one to be carried out in this paper!) 
Based on Belabas's data, Roberts \cite{R} conjectured the 
existence of a secondary term of order $X^{5/6}$ in $N^{\pm}_3(X)$ and this 
was proved (independently, and using different methods) by Bhargava, Shankar,
and Tsimerman \cite{BST}, and by Taniguchi and the second author \cite{TT}. 
Further details and references can be found in the survey papers above.

In the present paper we study cubic fields from a different angle. In 1954 
Cohn \cite{Cohn} studied {\itshape cyclic} cubic fields and proved that
\begin{equation}\label{eqn_cohn}
\sum_{K \ \textnormal{cyclic}} \frac{1}{\Disc(K)^s} = - \frac{1}{2} + \frac{1}{2} \bigg( 1 + \frac{1}{3^{4s}} \bigg) \prod_{p \equiv 1 \pmod 6}
\bigg(1 + \frac{2}{p^{2s}}\bigg)\;.
\end{equation}
To formulate a related question for noncyclic fields, fix a fundamental 
discriminant $D$. Given a noncyclic cubic field $K$, its Galois closure 
$\widetilde{K}$ has Galois group $S_3$ and hence contains a unique quadratic
subfield $k$, called the {\itshape quadratic resolvent}. For a fixed $D$, let
$\calF(\Q(\sqrt{D}))$ be the set of all cubic fields $K$ whose quadratic 
resolvent field is $\Q(\sqrt{D})$. For any $K \in \calF(\Q(\sqrt{D}))$ we have
$\Disc(K) = D f(K)^2$ for some positive integer $f(K)$, and we define 
\begin{equation}
\Phi_D(s) := \frac{1}{2} + \sum_{K \in \calF(\Q(\sqrt{D}))} \frac{1}{f(K)^s}\;,
\end{equation}
where the constant $1/2$ is added to simplify the final formulas.

Motivated by Cohn's formula \eqref{eqn_cohn}, we may ask if $\Phi_D(s)$ can 
be given an explicit form.

The answer is yes, as was essentially shown by A.~Morra and the first author 
in \cite{CM}, using Kummer theory. They proved a very general formula 
enumerating relative cubic extensions of any base field. However, this 
formula is rather complicated, and it is not in a form which is immediately
conducive to applications. In the present paper, we will show that this 
formula can be put in such a form when the base field is $\Q$; our formula 
(Theorem \ref{thm_main_cubic}) is similar to \eqref{eqn_cohn} but involves 
one additional Euler product for each cubic field of discriminant $-D/3$, 
$-3D$, or $-27D$.

\subsection{An application}
One application of our result is to enumerating $S_3$-sextic field extensions, i.e., sextic
field extensions $\wK$ which are Galois over $\Q$ with Galois group $S_3$. Suppose that $\wK$ is such a field,
where
$K$ and $k$ are the cubic and quadratic subfields respectively, the former being defined only up to isomorphism.
Then $k$ is the quadratic resolvent of $K$, and in addition to the formula $\Disc(K) = \Disc(k) f(K)^2$ we have
\begin{equation}
\Disc(\wK) = \Disc(K)^2 \Disc(k) = \Disc(k)^3 f(K)^4.
\end{equation}
so that our formulas may be used to count all such $\wK$ of bounded discriminant, starting from Belabas's tables \cite{Bel}
of cubic fields. Ours is not the only way to enumerate such $\wK$, but it is straightforward to implement and 
it seems to be (roughly) the most efficient. 

We implemented this algorithm using PARI/GP \cite{pari} to 
compute counts of $S_3$-sextic $\wK$ with $|\Disc(\wK)| < 10^{23}$. 
In Section \ref{sec_computations} we present our data, and the accompanying code is available from the second author's
website.
\\
\\
{\bf Outline of the paper.}
In Section \ref{sec_cubic_intro} we introduce our notation and give the main 
results. In Section \ref{sec_cubic_prelim} we summarize the work of Morra and
the first author \cite{CM}, and prove several propositions which will be 
needed for the proof of the main result. Our work relies heavily on work of 
Ohno \cite{Ohno} and Nakagawa \cite{N}, establishing an identity for binary 
cubic forms. In the same section, we give a result (Proposition \ref{case22})
which controls the splitting type of the prime $3$ in certain cubic 
extensions, and illustrates an application of Theorem \ref{thm_main_cubic}. 
Finally, in Section \ref{sec_cm} we prove Theorem \ref{thm_main_cubic}, using
the main theorem of \cite{CM}, recalled as Theorem \ref{theorem61}, as a 
starting point. In Section \ref{sec_examples} we give some
numerical examples which were helpful in double-checking our results, and in
Section \ref{sec_computations} we describe our computation of $S_3$-sextic fields.

\section*{Acknowledgments}
The authors would like to thank Karim Belabas, Franz Lemmermeyer,
Guillermo Mantilla-Soler, and Simon Rubinstein-Salzedo,
among many others, for helpful discussions related to the topic of this paper. We would especially like
to thank the anonymous referee of \cite{TT}, who (indirectly) suggested the application to counting
$S_3$-sextic fields.


\section{Statement of Results}\label{sec_cubic_intro}

We begin by introducing some notation. In what follows, by abuse of language 
we use the term ``cubic field'' to mean 
``isomorphism class of cubic number fields''.

\begin{definition}\label{def_omegal} 
Let $E$ be a cubic field. For a prime number $p$ we set
$$\om_E(p)=\begin{cases}
-1&\text{\quad if $p$ is inert in $E$\;,}\\
2&\text{\quad if $p$ is totally split in $E$\;,}\\ 
0&\text{\quad otherwise.}
\end{cases}$$
\end{definition}

\begin{remarks}\hfill

\begin{enumerate}
\item 
We have 
$\omega_E(p) = \chi(\sigma_p)$, where 
$\chi$ is the character of the standard representation of $S_3$, and 
$\sigma_p$ is the Frobenius element of $E$ at $p$.
\item
Note that we have $\om_E(p)=0$ if and only if $\leg{\Disc(E)}{p}\ne1$, and 
since in all cases that we will use we have $\Disc(E)=-D/3$, $-3D$, or $-27D$
for some fundamental discriminant $D$, for $p\ne3$ this is true if and only
if $\leg{-3D}{p}\ne1$. Thus, in Euler products involving the quantities
$1+\om_E(p)/p^s$ we can either include all $p\ne3$, or restrict to
$\leg{-3D}{p}=1$.
\end{enumerate}
\end{remarks}
\begin{definition}
\hfill
\begin{enumerate}\item Let $D$ be a fundamental discriminant
(including $1$). We let $D^*$ be the discriminant of the 
{\itshape mirror field} $\Q(\sqrt{-3D})$, so that $D^* = -3D$ if $3 \nmid D$ 
and $D^* = -D/3$ if $3 | D$.
\item For any fundamental discriminant $D$ we denote by $\rk_3(D)$ the
$3$-rank of the class group of the field $\Q(\sqrt{D})$.
\item For any integer $N$ we let $\LL_N$ be the set of cubic fields of 
discriminant $N$. We will use the notation $\LL_N$ only for $N=D^*$ or 
$N=-27D$, with $D$ a fundamental discriminant.
\item If $K_2=\Q(\sqrt{D})$ with $D$ fundamental we denote by $\FF(K_2)$ the
set of cubic fields with resolvent field equal to $K_2$, or equivalently, 
with discriminant of the form $Df^2$.
\item With a slight abuse of notation, we let
$$\LL_3(K_2)=\LL_3(D)=\LL_{D^*}\cup\LL_{-27D}.$$
\end{enumerate}\end{definition}

\begin{remark}
Scholz's theorem tells us that for $D<0$ we have
$0\le \rk_3(D)-\rk_3(D^*)\le1$ (or equivalently that for $D>0$ we have
$0\le \rk_3(D^*)-\rk_3(D)\le1$), and gives also a necessary and sufficient
condition for $\rk_3(D)=\rk_3(D^*)$ in terms of the fundamental unit of
the real field.
\end{remark}

\begin{definition}
As in the introduction, for any fundamental discriminant $D$ we define the
Dirichlet series
\begin{equation}
\Phi_D(s) := \frac{1}{2} + \sum_{K \in \calF(\Q(\sqrt{D}))} \frac{1}{f(K)^s}.
\end{equation}
\end{definition}

Our main theorem is as follows:

\begin{theorem}\label{thm_main_cubic} For any fundamental discriminant $D$ we 
have
\begin{equation}\label{eqn_main_cubic}
c_D\Phi_D(s)=\dfrac{1}{2}M_1(s)\prod_{\leg{-3D}{p}=1}\left(1+\dfrac{2}{p^s}\right)
+\sum_{E\in\LL_3(D)}M_{2,E}(s)\prod_{\leg{-3D}{p}=1}\left(1+\dfrac{\om_E(p)}{p^s}\right)\;,\end{equation}
where $c_D=1$ if $D = 1$ or $D < -3$, $c_D=3$ if $D = -3$ or $D > 1$, and
the $3$-Euler factors $M_1(s)$ and $M_{2,E}(s)$ are given in
the following table.
\end{theorem}

\bigskip

\centerline{
\begin{tabular}{|c||c|c|c|}
\hline
Condition on $D$ & $M_1(s)$ & $M_{2,E}(s),\ E\in\LL_{D^*}$ & $M_{2,E}(s),\ E\in\LL_{-27D}$\\
\hline\hline
$3\nmid D$       & $1+2/3^{2s}$ & $1+2/3^{2s}$        & $1-1/3^{2s}$\\
\hline
$D\equiv3\pmod9$ & $1+2/3^s$    & $1+2/3^s$           & $1-1/3^s$\\
\hline
$D\equiv6\pmod9$ & $1+2/3^s+6/3^{2s}$ & $1+2/3^s+3\om_E(3)/3^{2s}$ & $1-1/3^s$\\
\hline
\end{tabular}}

\bigskip

\begin{remarks}\hfill
\begin{enumerate}\item When $D\equiv3\pmod9$ we have 
$D^*\equiv2\pmod3$, so $3$ is partially split in any cubic field of 
discriminant $D^*$. It follows that when $E\in\LL_{D^*}$ we have 
$M_{2,E}(s)=1+2/3^s+3\om_E(3)/3^{2s}$ for all $D$ such that $3\mid D$.
\item When $3\nmid D$ there are no terms for $1/3^s$, in accordance with 
Proposition \ref{prop_disc_vals} below.
\item
In the terms involving $E\in\LL_{D^*}$ the condition $\leg{-3D}{p}=1$ can be
replaced by $p\ne3$ and even omitted altogether if $3\nmid D$, and in 
the terms involving $E\in\LL_{-27D}$ it can be omitted. 
\item
The case $D = 1$ is the formula \eqref{eqn_cohn} of Cohn mentioned previously,
and the case $D = -3$ was proved by Morra and the first author \cite{CM}. 
The paper \cite{CM} also proves \eqref{eqn_main_cubic} when $D < 0$ and 
$3 \nmid h(D)$, in which case $\calL_3(D) = \emptyset$. In her thesis 
\cite{M}, Morra also proves some special cases of an analogue of 
\eqref{eqn_main_cubic} for cubic extensions of imaginary quadratic fields.
Finally, one additional case of \eqref{eqn_main_cubic}
was proved in \cite{T_no_ep}, with an application to Shintani zeta functions.
\end{enumerate}
\end{remarks}

\section{Preliminaries}\label{sec_cubic_prelim}
We briefly summarize the work of \cite{CM}, and introduce some further 
notation which will be needed in the proof. We assume from now on that 
$D \neq 1, -3$; these cases are similar but simpler, and are already handled
in \cite{CM, M} (and the case $D = 1$ in \cite{Cohn}).

Suppose that $K/\Q$ is a cubic field of discriminant $D n^2$, where 
$D \not \in \{ 1, -3 \}$ is a fundamental discriminant, and 
let $N$ be the Galois closure of $K$. Then $N(\sqrt{-3})$ is a cyclic cubic 
extension of $L := \Q(\sqrt{D}, \sqrt{-3})$, and Kummer theory implies that
$N(\sqrt{-3}) = L(\al^{1/3})$ for some $\al \in L$. We write (following 
\cite[Remark 2.2]{CM})
\begin{equation}\label{eq:glq}
\Gal(L/\Q) = \{1, \ \tau, \ \tau_2, \ \tau \tau_2 \},
\end{equation}
where $\tau, \ \tau_2, \ \tau \tau_2$ fix $\sqrt{D}, \ \sqrt{-3}, \ \sqrt{-3D}$ 
respectively.

The starting point of \cite{CM} is a correspondence between such fields $K$
and such elements $\al$. In particular, isomorphism classes of such $K$ are 
in bijection with equivalence classes of elements 
$1 \neq \overline{\al} \in L^{\times}/(L^{\times})^3$,
with $\al$ identified with its inverse, such that 
$\al \tau'(\al) \in (L^{\times})^3$ for $\tau' \in \{ \tau, \tau_2 \}$. 
We express this by writing (as in \cite[Definition 2.3]{CM})
$\overline{\al} \in (L^{\times}/(L^{\times})^3)[T]$, where
$T \subseteq \F_3[\Gal(L/\Q)]$ is defined by $T = \{ \tau + 1, \tau_2 + 1 \}$,
and the notation $[T]$ means that $\overline{\al}$ is annihilated by $T$.

To go further, we introduce the following definition:

\begin{definition}\label{defsel} Let $k$ be a number field and $\ell$ be a
prime.
\begin{enumerate}\item We say that an element $\al\in k^*$ is an $\ell$-virtual
unit if $\al\Z_k=\q^\ell$ for some ideal $\q$ of $k$, or equivalently, if
$v_{\p}(\al)\equiv0\pmod{\ell}$ for all prime ideals $\p$, and we denote by
$V_{\ell}(k)$ the group of $\ell$-virtual units.
\item We define the $\ell$-Selmer group $S_{\ell}(k)$ of $k$ as 
$S_{\ell}(k)=V_{\ell}(k)/{k^*}^{\ell}$.
\end{enumerate}\end{definition}

Using this definition, it is immediate to see that the bijection described
above induces a bijection between fields $K$ as above, and triples
$(\mfa_0, \mfa_1, \overline{u})$ (up to equivalence with the triple
$(\mfa_1, \mfa_0, 1/\overline{u})$), satisfying the following:

\begin{enumerate}
\item $\mfa_0$ and $\mfa_1$ are coprime integral squarefree ideals of $L$ 
such that $\mfa_0 \mfa_1^2 \in (I/I^3)[T]$ (where $I$ is the group of 
fractional ideals of $L$), and $\overline{\mfa_0 \mfa_1^2} \in \Cl(L)^3$.
\item $\overline{u}\in S_3(L)[T]$, and $\overline{u}\ne1$ if
$\mfa_0=\mfa_1=\Z_L$.
\end{enumerate}

Indeed, given $\al$ such that $N(\sqrt{-3}) = L(\al^{1/3})$, we can write
uniquely $\al=\mfa_0\mfa_1^2\mfq^3$ with $\mfa_0$ and $\mfa_1$ coprime integral
squarefree ideals, and since 
$\overline{\al} \in (L^{\times}/(L^{\times})^3)[T]$ and the ideal class of
$\mfa_0\mfa_1^2$ is equal to that of $\mfq^{-3}$, the conditions on the
ideals are satisfied. Conversely, given a triple as above, we write 
$\mfa_0 \mfa_1^2 \mfq_0^3 = \al_0 \Z_L$ for some 
$\al_0 \in (L^* / (L^*)^3)[T]$ and some ideal $\mfq_0$. Then $K$ is the cubic
subextension of $L(\sqrt[3]{\al_0 u})$, for any lift $u$ of $\overline{u}$.

\smallskip

It is easy to see that $\mfa_0 \mfa_1 = \mfa_{\al} \Z_L$ for some ideal 
$\mfa_{\al}$ of $\Q(\sqrt{D})$, and the conductor
$\mff(K(\sqrt{D})/\Q(\sqrt{D}))$ is equal to $\mfa_{\al}$ apart from a
complicated $3$-adic factor. Furthermore,
$\mff(K(\sqrt{D})/\Q(\sqrt{D})) = f(K) \Z_{\Q(\sqrt{D})}$, and the Dirichlet
series for $\Phi_D(s)$ consists of a sum involving the norms of ideals 
$\mfa_0$ and $\mfa_1$ satisfying the conditions above. The condition 
$\overline{\mfa_0 \mfa_1^2} \in \Cl(L)^3$ may be detected by summing over 
characters of $\Cl(L)/\Cl(L)^3$, suggesting that cubic fields $K$ can be 
counted in terms of unramified abelian cubic extensions of $L$.

Due to the $3$-adic complications, the formula (Theorem 6.1 of \cite{CM}) in 
fact involves a sum over characters of the group
\begin{equation}\label{def:gb}
G_{\mfb} := \frac{\Cl_{\mfb}(L)}{(\Cl_{\mfb}(L))^3}[T]
\end{equation}
for $\mfb \in \calB:=\{ (1), (\sqrt{-3}), (3), (3 \sqrt{-3}) \}$.
More precisely, in the case considered here where the base field is $\Q$,
Theorem 6.1 of \cite{CM} specializes to the following 
(see also \cite{M})\footnote{Note that we slightly changed the definition of $F(\mfb,\chi,s)$ given in \cite{CM} when $\mfb=(1)$ and $3\mid D$.}:

\begin{theorem}\label{theorem61} If $D\notin\{ 1, -3\}$ we have
$$\Phi_D(s)=\dfrac{3}{2c_D}\sum_{\mfb\in\calB}A_{\mfb}(s)\sum_{\chi\in\widehat{G_{\mfb}}}\om_{\chi}(3)F(\mfb,\chi,s)\;,$$
where $c_D=1$ if $D<0$, $c_D=3$ if $D>0$,
the $A_{\b}(s)$ are given by the following table,

\medskip

\centerline{
\begin{tabular}{|c||c|c|c|c|}
\hline
Condition on $D$ & $A_{(1)}(s)$ & $A_{(\sqrt{-3})}(s)$ & $A_{(3)}(s)$ & $A_{(3\sqrt{-3})}(s)$\\
\hline\hline
$3\nmid D$       & $3^{-2s}$ & $0$        & $-3^{-2s-1}$ & $1/3$\\
\hline
$D\equiv3\pmod9$ & $0$    & $3^{-3s/2}$   & $3^{-s}-3^{-3s/2}$ & $(1-3^{-s})/3$\\
\hline
$D\equiv6\pmod9$ & $3^{-2s}$ & $3^{-3s/2}$ & $3^{-s}-3^{-3s/2}$ & $(1-3^{-s})/3$\\
\hline
\end{tabular}}

\smallskip
$$F(\mfb,\chi,s)=\prod_{\leg{-3D}{p}=1}\left(1+\dfrac{\om_{\chi}(p)}{p^s}\right)\;,$$
where if we write $p\Z_L=\mfc\tau(\mfc)$ (with $\mfc$ not necessarily prime),
we set\footnote{Note that this fixes a small mistake in 
the statement of Theorem 6.1 of \cite{CM}, where the condition is 
described as $\chi(\mfc) = \chi(\tau'(\mfc))$. The conditions are equivalent 
whenever $p$ is a cube modulo $\mfb$; if $\mfb = (3 \sqrt{-3})$ and 
$p \not \equiv \pm 1 \ (\textmod \ 9)$, then $p$ and $p \tau'(p)$ are not 
cubes in $\Cl_{\mfb}(L)$ for $\tau' \in \{\tau, \tau_2\}$, and so the class 
of $p$ is not in $G_{\mfb}$.} for $p\ne3$:
$$\om_{\chi}(p)=\begin{cases}
2&\text{\quad if $\chi(\tau(\c)/\c) = 1$}\;,\\
-1&\text{\quad if $\chi(\tau(\c)/\c) \neq 1$}\;,\end{cases}$$
and for $p=3$:
$$\om_{\chi}(3)=\begin{cases}
1&\text{\quad if $\mfb\ne(1)$ or $\mfb=(1)$ and $3\nmid D$}\;,\\
2&\text{\quad if $\mfb=(1)$, $3\mid D$, and $\chi(\tau(\c)/\c) = 1$}\;,\\
-1&\text{\quad if $\mfb=(1)$, $3\mid D$, and $\chi(\tau(\c)/\c) \neq 1$}\;.\end{cases}$$
\end{theorem}

\begin{proof} We briefly explain how this follows from Theorem 6.1 of
\cite{CM}. Warning: in the present proof we use the notation of \cite{CM} 
which conflicts somewhat with that of the present paper. All the
definition, proposition, and theorem numbers are those of \cite{CM}.

\begin{itemize}
\item We have $k=\Q$ so $[k:\Q]=1$, so $3^{(3/2)[k:\Q]s}=3^{3s/2}$.
\item By Definition 3.6 we have ${\calP}_3=\{3\}$ if $3\nmid D$ and 
$\emptyset$ if $3\mid D$, so $\prod_{p\in{\calP}_3}p^{s/2}=3^{s/2}$ if $3\nmid D$
and $1$ if $3\mid D$.
\item We have $k_z=\Q(\sqrt{-3})$, $K_2=\Q(\sqrt{D})$, and 
$L=\Q(\sqrt{-3},\sqrt{D})$, so by Lemma 5.4 we have
$|(U/U^3)[T]|=3^{r(U)}$ with $r(U)=2+0-1-\delta_{D>0}$, where $\delta$ is
the Kronecker symbol, hence $3^{r(U)}=3/c_D$ with the notation of our
theorem.
\item By Definition 4.4, if $3\nmid D$ we have 
$\lceil \N(\mfb)\rceil=(1,*,3,3^2)$ while if $3\mid D$ we have 
$\lceil \N(\mfb)\rceil=(1,3^{1/2},3,3^{3/2})$ for
$\mfb=((1),(\sqrt{-3}),(3),(3\sqrt{-3}))$ respectively. Note that we use the
convention of Definition 4.1, so that $\lceil \N(\mfb)\rceil$ can be the
square root of an integer.
\item By Definition 4.4 we have $\N({\mathfrak r}^e(\mfb))=1$ unless 
$\mfb=(1)$ and $3\mid D$, in which case $\N({\mathfrak r}^e(\mfb))=3^{1/2}$
(where we again use the convention of Definition 4.1).
\item By Proposition 2.10 we have $\calD_3=\emptyset$ (hence 
${\mathfrak d}_3=1$) unless $D\equiv6\pmod9$, in which case $\calD_3=\{3\}$
(hence ${\mathfrak d}_3=3$), and for $p\ne3$ we have $p\in\calD$ if and only if
$\leg{-3D}{p}=1$. In particular ${\mathfrak r}^e(\mfb)\nmid{\mathfrak d}_3$ if
and only if $\mfb=(1)$ and $D\equiv3\pmod9$.
\item By Definition 4.5, if $3\nmid D$ we have $P_{\mfb}(s)=(1,*,-3^{-s},1)$ 
while if $3\mid D$ we have $P_{\mfb}(s)=(1,3^{-s/2},3^{-s/2}-3^{-s},1-3^{-s})$
for $\mfb=((1),(\sqrt{-3}),(3),(3\sqrt{-3}))$ respectively\footnote{Note that
there is a misprint in Definition 4.5 of \cite{CM}: when $e(\p_z/\p)=1$
and $b=0$, we must set $Q((p\Z_{K_2})^b,s)=1$ and not $0$. The vanishing of
certain terms in the final sum comes from the condition 
${\mathfrak r}^e(\mfb)\mid{\mathfrak d}_3$ of the theorem.}.
\item By Lemma 5.6, if $3\nmid D$ we have 
$|(Z_{\mfb}/Z_{\mfb}^3)[T]|=(1,*,3,3)$ while if $3\mid D$ we have 
$|(Z_{\mfb}/Z_{\mfb}^3)[T]|=(1,1,1,3)$ for
$\mfb=((1),(\sqrt{-3}),(3),(3\sqrt{-3}))$ respectively.
\end{itemize}
(In the above we put $*$ whenever $3\nmid D$ and $\mfb=(\sqrt{-3})$ since this
case is impossible.)

The theorem now follows immediately from Theorem 6.1 of \cite{CM}.
\end{proof}

For future reference, note the following lemma whose trivial proof is
left to the reader:

\begin{lemma}\label{lemomchi} Let $\chi$ be a cubic character, and as above
write $p\Z_L=\mfc\tau(\mfc)$. The following conditions are equivalent:
\begin{enumerate}
\item $\chi(\tau(\c)/\c) = 1$.
\item $\om_{\chi}(p)=2$.
\item $\chi(p\mfc)=1$.
\end{enumerate}
If these conditions are not satisfied we have $\om_{\chi}(p)=-1$.
\end{lemma}

To proceed further, we need to compute the size of the groups $G_{\mfb}$ and
to reinterpret the conditions involving $\chi$ as conditions involving
the cubic field associated to each pair $(\chi,\ov{\chi})$.

In what follows we write 
\begin{equation}\label{def:ha}
H_{\mfa} := \frac{\Cl_{\mfa}(D^*)}{(\Cl_{\mfa}(D^*))^3}[1 + \tau]\;,
\ \ \
H'_{\mfa} := \frac{\Cl_{\mfa}(D)}{(\Cl_{\mfa}(D))^3}[1 + \tau']\;,
\end{equation}
where $\tau$ and $\tau'$ are the nontrivial elements of 
$\Gal(\Q(\sqrt{D^*})/\Q)$ and $\Gal(\Q(\sqrt{D})/\Q)$ respectively, 
and $\Cl_{\mfa}(N)$ is shorthand for $\Cl_{\mfa}(\Q(\sqrt{N}))$. This $\tau$ 
is the restriction of $\tau \in \Gal(L/\Q)$ (see \eqref{eq:glq}) to 
$\Q(\sqrt{D^*})$, and we regard $\tau$ as an automorphism of both $L$ and of 
$\Q(\sqrt{D^*})$ (and $\tau'$ is the restriction of $\tau_2$, but we prefer
calling it $\tau'$).

\begin{proposition}\label{prop_g_size} We have
\begin{equation}\label{eqn_g1}
G_{\mfb} \simeq H_{(a)},
\end{equation}
where
\begin{itemize}
\item $a = 1$ if $\mfb = (1)$ or $(\sqrt{-3})$, or if $\mfb = (3)$ and $3 | D$;
\item $a = 3$ if $\mfb = (3)$ or $(3 \sqrt{-3})$, and $3 \nmid D$;
\item $a = 9$ if $\mfb = (3 \sqrt{-3})$ and $3 | D$.
\end{itemize}
\end{proposition}

\begin{remarks}\hfill 
\begin{enumerate}
\item
Later we will associate a cubic field of discriminant $-D/3$, $-3D$, or $-27D$
to each pair of conjugate nontrivial characters of $G_{\mfb}$. Propositions 
\ref{prop_g_size} and \ref{prop_count_cf} will show that we obtain all such 
fields in this manner.
\item
Propositions \ref{prop_g_size},  \ref{prop_count_cf}, and 
\ref{prop_disc_vals} imply equalities among $|H_{(3^n)}|$ for different values
of $n$. In particular, $|H_{(3^n)}| = |H_{(9)}|$ for $n > 2$, and if $3 | D$ 
then $|H_{(3)}| = |H_{(1)}|$ as well. 
\end{enumerate}
\end{remarks}

\begin{proof}
For $\mfb = (1)$, $G_{\mfb}$ is just $\big( \Cl(L) / \Cl(L)^3 \big)[T]$. This
case may be handled with the others, but for clarity we describe it first.

By arguments familiar in the proof of the Scholz reflection principle (see,
e.g. \cite{Wa}, p. 191) we have
\begin{equation}\label{eqn_scholz}
\Cl(L) / \Cl(L)^3 \simeq \Cl(D) / \Cl(D)^3 \oplus \Cl(D^*) / \Cl(D^*)^3
\end{equation}
as $\Gal(L/\Q)$-modules. Since $\tau$ acts trivially on $\Q(\sqrt{D})$ we have
$(\Cl(D)/\Cl(D)^3)[1+\tau]=1$ hence $(\Cl(D)/\Cl(D)^3)[T]=1$, and since
$\tau_2$ acts nontrivially on $\Cl(D^*)/\Cl(D^*)^3$, $1+\tau_2$ acts as the 
norm, which annihilates the class group, so
$(\Cl(D^*)/\Cl(D^*)^3)[T]=\Cl(D^*)/\Cl(D^*)^3[1+\tau]$, so finally
\begin{equation}
\big(\Cl(L)/\Cl(L)^3\big)[T] = \Cl(D^*) / \Cl(D^*)^3[1+\tau] = H_{(1)}\;,
\end{equation}
as desired (note that since $\tau$ also acts nontrivially we have in fact
$\Cl(D^*) / \Cl(D^*)^3[1+\tau]=\Cl(D^*) / \Cl(D^*)^3$).

Now suppose that $\mfb$ is equal to $(\sqrt{-3})$, $(3)$, or $(3 \sqrt{-3})$. 
As $\tau$ and $\tau_2$ both act nontrivially on $G_{\mfb}$, 
$\tau \tau_2$ acts trivially. Moreover
$(1 + \tau \tau_2)/2 \in \F_3[\Gal(L/\Q)]$
is an idempotent, so the elements of $G_{\mfb}$ are those that may be represented by
an ideal of the form $\mfa \tau \tau_2(\mfa)$, which is necessarily of the form 
$\mfa' \Z_L$ for an ideal $\mfa'$ of $\Q(\sqrt{D^*})$. When $\mfb = (1)$ this yields an 
isomorphism $G_{(1)} \xrightarrow{\sim} H_{(1)}$, following \eqref{eqn_scholz}. When
$\mfb \ne (1)$, this yields an isomorphism 
$G_{\mfb} \xrightarrow{\sim} H_{\mfa'}$, where $\mfa' := \mfb \cap \Z_{\Q(\sqrt{D^*})}$.
In this case the $(1 + \tau)$-invariance is no longer automatic.

For convenience we write $F := \Q(\sqrt{D^*})$ in the remainder of the proof.
The ideal $\mfb \cap \Z_F$ is simple to determine. If $3 | D$, then $3$ is unramified in $F$, and so $\mfb \cap \Z_F$ is equal to $(1)$, $(3)$, $(3)$, $(9)$
for $\mfb = (1)$, $(\sqrt{-3})$, $(3)$, $(3 \sqrt{-3})$ respectively, as desired. Moreover,
in this case Propositions \ref{prop_count_cf} and \ref{prop_disc_vals} below
imply that $H_{(3)} \simeq H_{(1)}$ (there are no cubic fields whose discriminants
have $3$-adic valuation $2$).

If $3 \nmid D$, then write $(3) = \mfp^2$ in $\Z_F$, and for $\mfb = (1)$, 
$(\sqrt{-3})$, $(3)$, $(3 \sqrt{-3})$, $\mfb \cap \Z_F$ is equal to 
$(1)$, $\mfp$, $(3)$, $3 \mfp$ respectively.
We write down the {\itshape ray class group exact sequence}
\begin{equation}\label{eqn_rcges}
1 \rightarrow
\Z_F^{\times} / \Z_F^{\mfa'}
\rightarrow
(\Z_F / \mfa')^{\times}
\rightarrow
\Cl_{\mfa'}(F)
\rightarrow
\Cl(F)
\rightarrow
1,
\end{equation}
where $\Z_F^{\mfa'}$ is the subgroup of units congruent to $1$ modulo $\mfa'$. We
take $3$-Sylow subgroups and take negative eigenspaces for the action of $\tau$ 
(i.e., write each $3$-Sylow subgroup $A$ as $A^+ \oplus A^-$, where 
$A^{\pm} := \{x \in A: \tau(x) = x^{\pm 1} \}$, and take $A^-$); 
these operations preserve exactness. If $A$ is the $3$-Sylow subgroup of 
$\Cl_{\mfa'}(F)$, then $H_{\mfa'}$ is isomorphic to 
$(A/A^3)^- \simeq A^- / (A^-)^3$.

For $\mfb = (1)$ or $(3)$, this finishes the proof. For $\mfb = (\sqrt{-3})$, 
the $3$-part of $(\Z_F / \mfp)^{\times}$ is trivial;
hence, the $3$-Sylow subgroups of $\Cl(F)$ and $\Cl_{\mfp}(F)$ are isomorphic, so
$H_{\mfp} \simeq H_{(1)}$.

For $\mfb = (3 \sqrt{-3})$, the $3$-Sylow subgroup of $(\Z_F/\mfp^3)^{\times}$ is
larger than that of $(\Z_F/3)^{\times}$ by a factor of $3$; however, the same is 
true of the positive eigenspace and hence not of the negative 
eigenspace. Therefore $H_{\mfp^3} \simeq H_{(3)}$. 

\end{proof}

We now state a well-known formula for counting cubic fields in terms of ray
class groups.

\begin{proposition}\label{prop_count_cf}
If $D$ is a fundamental discriminant other than $1$, $c$ is any nonzero 
integer, and $H'_{\mfa}$ is defined as in \eqref{def:ha}, then we have
\begin{equation}
\sum_{d | c} |\calL_{D d^2}| = \frac{1}{2} \Big( \big| H'_{(c)} \big| - 1\Big).
\end{equation}
\end{proposition}
\begin{proof}
This is a combination of (1.3) and Lemma 1.10 of Nakagawa \cite{N}, and it is also
a fairly standard fact from class field theory. The idea is that cubic extensions
of $\Q(\sqrt{D^*})$ whose conductor divides $(c)$ correspond to subgroups of 
$\Cl_{(c)}(D^*)$ of index $3$, and that such a cubic extension descends to a cubic
extension of $\Q$ if and only if it is in the kernel of $1 + \tau$.
\end{proof}

We also have the following counting result, which relies on the deeper part 
of the work of Nakagawa \cite{N} and Ohno \cite{Ohno}.
\begin{proposition}\label{prop_no}
Let $D$ be a fundamental discriminant, and set $r=\rk_3(D^*)$.
\begin{enumerate}[(1.)]

\item
Assume that $D < -3$. Then we have
  \begin{displaymath}
  (|\calL_{D^*}|, |\calL_{-27D}|) = 
   \left\{
     \begin{array}{ll}
       ((3^r - 1)/2, 3^r) & {\rm if } \ \rk_3(D) = r + 1, \\
       ((3^r - 1)/2, 0) & {\rm if } \ \rk_3(D) = r. \\
     \end{array}
   \right.
\end{displaymath} 
In  either case, $|\calL_3(D)| = (3^{\rk_3(D)} - 1)/2$.

\item
Assume that $D > 1$. Then we have
  \begin{displaymath}
  (|\calL_{D^*}|, |\calL_{-27D}|) = 
   \left\{
     \begin{array}{ll}
       ((3^r - 1)/2, 0) & {\rm if } \ \rk_3(D) = r - 1, \\
       ((3^r - 1)/2, 3^r) & {\rm if } \ \rk_3(D) = r. \\
     \end{array}
   \right.
\end{displaymath} 
In  either case, $|\calL_3(D)| = (3^{\rk_3(D) + 1} - 1)/2$.

\end{enumerate}

\end{proposition}

\begin{proof}
The formulas for $|\calL_{D^*}|$ follow from class field theory, as these count
unramified cyclic cubic extensions of $\Q(\sqrt{-3D})$, which are in bijection 
with subgroups of $\Cl(-3D)$ of index $3$. It therefore
suffices to prove the stated formulas for $|\calL_3(D)|$, and these follow from 
work of Nakagawa \cite{N}.

Recalling the notation in \eqref{def:ha},
Nakagawa proved \cite[Theorem 0.4]{N} that if $D < 0$, then
\begin{equation}
|H'_{(1)}| = |H_{(a)}|,
\end{equation}
where $a = 3$ if $3 \nmid D$ and $a = 9$ if $3 | D$; and if $D > 0$ then
\begin{equation}
3 |H'_{(1)}| = |H_{(a)}|
\end{equation}
with the same $a$.

By Proposition \ref{prop_count_cf}, these formulas are 
equivalent to the stated formulas for $|\calL_3(D)|$.
\end{proof}

\begin{proposition}\label{prop_disc_vals}
If $K$ is a cubic field then $v_3(\Disc(K))$ can only be equal to $0$, $1$, 
$3$, $4$, and $5$ in relative proportions 
$81/117, \ 27/117, \ 6/117, \ 2/117,$ and $1/117$, when the fields are
ordered by increasing absolute value of their discriminant.
\end{proposition}

\begin{proof}
The proof that $v_3(\Disc(K))$ can take only the given values is classical; 
see Hasse \cite{Has}. The proportions follow from the proof of the 
Davenport-Heilbronn theorem. A convenient reference is Section 6.2 of \cite{TT}, 
where a table of these proportions is given in the context of 
``local specifications''; these proportions also appear
(in slightly less explicit form) in the earlier related literature.
\end{proof}

Before proceeding to the proof of Theorem \ref{thm_main_cubic}
we give the following application:

\begin{proposition}\label{case22}\hfill\begin{enumerate}\item If $D < 0$ and
$(\rk_3(D),\rk_3(D^*))=(2,1)$, or $D > 0$ and $(\rk_3(D),\rk_3(D^*))=(1,1)$
there exist a unique cubic field of discriminant $D^*$ and three cubic fields 
of discriminant $-27D$.
\item If $D < 0$ and $(\rk_3(D),\rk_3(D^*))=(2,2)$, or $D > 0$ and
$(\rk_3(D),\rk_3(D^*))=(1,2)$ there exist four cubic fields of discriminant $D^*$ 
and no cubic field of discriminant $-27D$.

In addition, if $3\nmid D$ then $3$ is partially ramified in the four cubic
fields, if $D\equiv3\pmod9$ then $3$ is partially split in the four cubic
fields, and if $D\equiv6\pmod9$ then $3$ is totally split in one of the four
cubic fields and inert in the three others.
\end{enumerate}\end{proposition}

\begin{proof} The first statements are special cases of Proposition \ref{prop_no}, the
behavior of $3$ when $3\nmid D$ is classical (see \cite{Has}), and when
$D\equiv3\pmod9$ the last statement is trivial since $D^*\equiv2\pmod3$. 

For the case
$D\equiv6\pmod9$ we use Theorem \ref{thm_main_cubic}. Writing out 
the 3-part of the theorem for the discriminant $D$, we see that
\begin{equation}
| \calL_{81D} | = 3 \biggl(1 + \sum_E  \om_E(3) \biggr)\;,
\end{equation}
where the sum ranges over the cubic fields $E$ of discriminant $-D/3$, and 
$\om_E(3)$ is equal to $2$ if $3$ is totally split in $E$, and $-1$ if $3$
is inert in $E$. Therefore, if $3$ is totally split in $0$, $1$, $2$, $3$, or 
$4$ of these fields then $|\calL_{81D}|$ is equal to $-9$, $0$, $9$, $18$, or
$27$. Obviously we can rule out the first possibility.

We first observe that $| \calL_{9D}| = 9$, again by Theorem \ref{thm_main_cubic}. 
By Proposition \ref{prop_count_cf},
\begin{equation}
| \calL_{9D} | = 9 = \frac{1}{2} \Big( H'_{(3)} - H'_{(1)} \Big), \ \ \
| \calL_{81D} | = \frac{1}{2} \Big( H'_{(9)} - H'_{(3)} \Big).
\end{equation}
By assumption, $|H'_{(1)}| = 9$ and so $|H'_{(3)}| = 27$.
Therefore, either $|H'_{(9)}| = 81$ and $| \calL_{81D} | = 27$ or
 $|H'_{(3)}| = 27$ and $| \calL_{81D} | = 0$.

To rule out the former possibility, we again consider the exact sequence 
\eqref{eqn_rcges}, with $F = \Q(\sqrt{D^*})$ replaced by $\Q(\sqrt{D})$ and $\tau$ replaced
by $\tau'$, and take $3$-Sylow subgroups and $(1 + \tau')$-invariants 
(preserving exactness).
The $3$-rank of $(\Z_{\Q(\sqrt{D})}/ (9))^{\times}[1 + \tau]$ is equal to $1$, and so the 
$3$-rank of $\Cl_{(9)}(D)$ is at most $1$ more than that of $\Cl(D)$. In other words, 
$|H'_{(9)}| \leq 3 |H'_{(1)}|$, but we saw previously that $|H'_{(3)}| = 3 |H'_{(1)}|$, so
$|H'_{(9)}| = 3 |H'_{(1)}|$ and $|\calL_{81D}| = 0$ as desired.
\end{proof}

\section{Proof of Theorem \ref{thm_main_cubic}}\label{sec_cm}
Theorem \ref{thm_main_cubic} follows from a more general result of Morra and 
the first author (Theorem 6.1 of \cite{CM} and Theorem 1.6.1 of \cite{M}), 
given above as Theorem \ref{theorem61} in our case where the base field
is $\Q$. To each character of the groups $G_{\mfb}$ we use class field 
theory and Galois theory to uniquely associate a cubic field $E$. Some 
arithmetic involving discriminants, as well as a comparison to our earlier 
counting formulas, proves that these fields $E$ range over all fields in
$\calL_3(D)$. Finally, we apply Theorem \ref{theorem61} to obtain the correct
Euler product for each $E$.

This section borrows from the first author's work in \cite{T_no_ep}, which
established a particular case of Theorem \ref{thm_main_cubic} for an 
application to Shintani zeta functions.

\subsection{Construction of the Fields $E$.}
We refer to the beginning of Section
\ref{sec_cubic_prelim} for our notation and setup.
The contribution of the trivial characters occurring in Theorem \ref{theorem61}
being easy to compute (see below; it has also been computed in \cite{CM}),
we must handle the {\itshape nontrivial} characters.

We relate these characters to cubic fields by means of the following.
\begin{proposition}\label{prop_cubic_bij}
For each $\mfb \in \calB$ there is a bijection between the set of pairs of 
nontrivial characters $(\chi, \overline{\chi})$ of $G_{\mfb}$ and the 
following sets of cubic fields $E$:
\begin{itemize}
\item
If $\mfb = (1)$ or $(\sqrt{-3})$, or if $\mfb = (3)$ and $3 | D$, then all 
$E \in \calL_{D^*}$.
\item
If $\mfb = (3 \sqrt{-3})$, or if $\mfb = (3)$ and $3 \nmid D$, then all 
$E \in \calL_3(D) = \calL_{D^*} \cup \calL_{-27D} $.
\end{itemize}
Moreover, for each prime $p$ with $\big( \frac{-3D}{p} \big) = 1$, write
$p \Z_L = \mfc \tau(\mfc)$ as in \cite{CM}, where $\mfc$ is not necessarily prime, and recall from Theorem \ref{theorem61} and Lemma \ref{lemomchi} the 
definition of $\om_{\chi}(p)$. Under our bijection, the following conditions
are equivalent:
\begin{itemize}
\item
$\chi(p \mfc) = 1$.
\item
The prime $p$ splits completely in $E$.
\item
$\om_{\chi}(p) = 2$.
\end{itemize}
If these conditions are not satisfied then $p$ is inert in $E$ and 
$\om_{\chi}(p) = -1$.
\end{proposition}

\begin{proof}
Define\footnote{We have followed the notation of \cite{CM} where practical, 
but the notations $G'_{\mfb}, \ G''_{\mfb}, \ E, \ E_1$ are 
used for the first time here and do not appear in \cite{CM}.} 
$G'_{\mfb} := \Cl_{\mfb}(L)/\Cl_{\mfb}(L)^3$, so that $G'_{\mfb}$ is
a $3$-torsion group containing $G_{\mfb}$. We have a canonical decomposition 
of $G'_{\mfb}$ into four eigenspaces for the actions of $\tau$ and $\tau_2$, 
and we write 
\begin{equation}\label{eqn_nci}
G'_{\mfb} \simeq G_{\mfb} \times G''_{\mfb},
\end{equation}
where $G''_{\mfb}$ is the direct sum of the three eigenspaces other than 
$G_{\mfb}$. Note that $G''_{\mfb}$ will contain the classes of all principal
ideals generated by rational integers coprime to $3$; any such class in the 
kernel of $T$ will necessarily be in $\Cl_{\mfb}(L)^3$.

For any nontrivial character $\chi$ of $G_{\mfb}$, let $B$ be its kernel, 
which has index $3$. We extend $\chi$ to a character $\chi'$
of $G'_{\mfb}$ by setting $\chi(G''_{\mfb}) = 1$, and write 
$B' := \Ker(\chi') = B \times G''_{\mfb} \subseteq G'_{\mfb}$, so that
$B'$ has index $3$ in $G'_{\mfb}$ and is uniquely determined by $\mfb$ and 
$\chi$. By class field theory, there is a unique abelian extension $E_1/L$ for 
which the Artin map induces an isomorphism $G'_{\mfb}/B' \simeq \Gal(E_1/L)$,
and it must be cyclic cubic, since $G'_{\mfb}/B'$ is. The uniqueness forces 
$E_1$ to be Galois over $\Q$, since the groups $G'_{\mfb}$ and $B'$ are 
preserved by $\tau$ and $\tau_2$ and hence by all of $\Gal(L/\Q)$. For each 
fixed $\mfb$, we obtain a unique such field $E_1$ for each distinct pair of 
characters $\chi, \overline{\chi}$, but we may obtain the same field $E_1$ for 
different values of $\mfb$.

We have $\Gal(E_1/\Q) \simeq S_3 \times C_2$: $\tau$ and $\tau_2 \in \Gal(E_1/L)$ 
both act nontrivially (elementwise) on $G_{\mfb}$ and preserve $B$, and hence 
act nontrivially on $G_{\mfb}/B$ and $G'_{\mfb'}/B'$. Under the Artin map 
this implies that $\tau$ and $\tau_2$ both act nontrivially on $\Gal(E_1/L)$ by 
conjugation, so that $\tau \tau_2$ commutes with $\Gal(E_1/L)$. As $\tau \tau_2$
fixes $\Q(\sqrt{-3D})$, this implies that $E_1$ contains a cubic extension 
$E/\Q$ with quadratic resolvent $\Q(\sqrt{-3D})$, which is unique up to 
isomorphism. Any prime $p$ which splits in $\Q(\sqrt{-3D})$ must either be 
inert or totally split in $E$.
\\
\\
We now prove that the fields $E$ which occur in this construction are those described by the proposition.
Since the quadratic resolvent
of any $E$ is $\Q(\sqrt{-3D})=\Q(\sqrt{D^*})$ with $D^*$ fundamental, we have
$\Disc(E) = r^2D^* $ for some integer $r$ divisible only by $3$ and prime 
divisors of $D^*$.

No prime $\ell > 3$ can divide $r$, because $\ell^3$ cannot divide the 
discriminant of any cubic field. Similarly $2$ cannot divide $r$, 
since if $2 | D$ then $4 | D$, but $16$ cannot divide the discriminant of a 
cubic field. Therefore $r$ must be a power of $3$. Since the $3$-adic 
valuation of a cubic field discriminant is never larger than $5$, $r$ must be 
$1$, $3$, or $9$, and by Proposition \ref{prop_disc_vals} we cannot
have $r=3$ if $3\nmid D^*$, in other words if $3\mid D$. It follows that
$E$ must have discriminant $D^*$, $-27D$, or $-243D$. 

We further claim that $\Disc(E) \neq -243D$. To see this, we apply the 
formula
\begin{equation}\label{eqn_disce}
\Disc(E_1) = \pm \Disc(L)^3 \N_{L / \Q} ( \mathfrak{d}(E_1/L))
\end{equation}
and bound $v_3(\Disc(E_1))$. We see that $v_3(\Disc(L)) = 2$ (by direct 
computation, or by a formula similar to \eqref{eqn_disce}). Moreover, the 
conductor of $E_1/L$ divides $(3 \sqrt{-3})$, and therefore $\mathfrak{d}(E_1/L)$
divides $(27)$. The norm of this ideal is $3^{12}$, and putting all of this 
together we see that $v_3(\Disc(E_1)) \leq 18$.

We also have 
\begin{equation}\label{eqn_disce2}
\Disc(E_1) = \pm \Disc(E)^4 \N_{E / \Q} ( \mathfrak{d}(E_1/E)),
\end{equation}
so that $v_3(\Disc(E)) < \frac{18}{4}$, and in particular 
$v_3(\Disc(E)) \neq 5$ as desired.

Therefore, in all cases $E$ must have discriminant $D^*$ or $-27D$. Moreover,
similar comparisons of \eqref{eqn_disce} and \eqref{eqn_disce2} show that if 
$\mfb \in \{(1), (\sqrt{-3}) \}$, or if $\mfb = (3)$ and $3 | D$, then $E$ 
must have discriminant $D^*$.

We have therefore associated a unique $E$ to each pair 
$(\chi, \overline{\chi})$ as in the proposition, and it follows from
Propositions \ref{prop_count_cf} and \ref{prop_g_size} that we obtain all such
$E$ in this manner.
\\
\\
Finally, we prove the second part of the proposition. The statements
concerning $\om_{\chi}(p)$ and $\chi(p\mfc)$ follow from Lemma
\ref{lemomchi}. We show the equivalence of the first and second statement.

We extend $\chi$ to a character $\chi'$ of $G'_{\mfb}$ as defined previously,
so that $\chi'(p)$ is defined and equal to 1. Therefore, $\chi(p \mfc) = 1$ if
and only if $\chi'(\mfc) = 1$, and we must show that this is true if and only 
if $p$ splits completely in $E$. 

Suppose first that $\mfc$ is prime in $\Z_L$. By class field theory, 
$\chi'(\mfc) = 1$ if and only if $\mfc$ splits completely in $E_1/L$, which 
happens if and only if $(p)$ splits into six ideals in $E_1$, which happens if
and only if $p$ splits completely in $E$.

Suppose now that $\mfc = \mfp \tau \tau_2(\mfp)$ in $L$. 
Since $\mfp$ and $\tau \tau_2(\mfp)$ represent the same element of
$G'_{\mfb}/B'$ they have the same Frobenius element in $E_1/L$, hence
since $\chi'$ is a cubic character it follows that $\chi'(\mfc)=1$ if and only
if $\chi'(\mfp)=1$. By class field theory this is true if and only if $\mfp$ 
splits completely in $E_1/L$, in which case $(p)$ splits into twelve ideals in
$E_1$; for this it is necessary and sufficient that $p$ split completely in 
$E$.\end{proof}

\subsection{Putting it all Together}
Applying Proposition \ref{prop_cubic_bij} we may regard the formula of Theorem
\ref{theorem61} as a sum over cubic fields. 
We now divide into cases $3 \nmid D$ and $D \equiv 3, 6 \ (\textmod \ 9)$.
The main terms, corresponding to the trivial characters of $G_{\mfb}$, 
contribute
\begin{equation}
\dfrac{3}{2c_D}\sum_{\mfb\in\calB}\om_1(3)A_{\mfb}(s)=\dfrac{1}{2c_D}M_1(s)\prod_{\leg{-3D}{p}=1}\left(1+\dfrac{2}{p^s}\right)
\end{equation}
of Equation \ref{eqn_main_cubic}. These terms have also been given in
\cite{CM} and \cite{M}.

It remains to handle the contribution of the nontrivial characters.

Assume first that $3 \nmid D$. In this case by Theorem \ref{theorem61} 
\begin{multline}
c_D\Phi_D(s)=\frac{3}{2} \cdot \Bigg[
3^{-2s}\sum_{\chi \in \widehat{G_{(1)}}} F((1), \chi, s) - 3^{-2s-1} \sum_{\chi \in \widehat{G_{(3)}}} F((3), \chi, s)\\
+ \frac{1}{3}\sum_{\chi \in \widehat{G_{(3 \sqrt{-3})}}} F((3\sqrt{-3}), \chi, s)) \Bigg]\;,
\end{multline}
and the calculations above show that 
\begin{equation}
F(\mfb, \chi, s) = \prod_{ \big( \frac{-3D}{p} \big) = 1} \bigg(1 + \frac{\om_{E}(p)}{p^s} \bigg)\;,
\end{equation}
where $E$ is the cubic field associated to $\chi$. Each field $E$ of 
discriminant $D^*$ contributes twice (each character yields the same field as 
its inverse) to each of the three sums above, and each field of discriminant 
$-27D$ contributes twice to each of the last two. We obtain a contribution
of $1 + 2 \cdot 9^{-s}$ for each field of discriminant 
$D^* = -3D$, and of $1 - 9^{-s}$ for each field of discriminant 
$-27D$. This is the assertion of the theorem.

Assume now that $D \equiv 3 \ (\textmod \ 9)$. Then
\begin{multline}
c_D\Phi_D(s)=\frac{3}{2} \cdot \Bigg[
3^{-3s/2}\sum_{\chi \in \widehat{G_{(\sqrt{-3})}}} F((\sqrt{-3}), \chi, s) + \big( 3^{-s} - 3^{-3s/2} \big) \sum_{\chi \in \widehat{G_{(3)}}} F((3), \chi, s)\\
 + \frac{1}{3} \big(1 - 3^{-s} \big) \sum_{\chi \in \widehat{G_{(3 \sqrt{-3})}}} F((3\sqrt{-3}), \chi, s)) \Bigg]\;.
\end{multline}
The first two sums are over fields of discriminant $D^* = -D/3$, and the last
sum also includes fields of discriminant $-27D$. We obtain a contribution of 
$1 + 2 \cdot 3^{-s}$ for each field of discriminant $D^*$, and
of $1 - 3^{-s}$ for each field of discriminant $-27D$, in accordance with the
theorem.

Finally, assume that that $D \equiv 6 \ (\textmod \ 9)$. Then
\begin{multline}
c_D\Phi_D(s)=\frac{3}{2} \cdot \Bigg[
3^{-2s} \sum_{\chi \in \widehat{G_{(1)}}} \om_{\chi}(3) F((1), \chi, s) + 
3^{-3s/2}\sum_{\chi \in \widehat{G_{(\sqrt{-3})}}} F((\sqrt{-3}), \chi, s) + \\ 
\big( 3^{-s} - 3^{-3s/2} \big) \sum_{\chi \in \widehat{G_{(3)}}} F((3), \chi, s) + 
\frac{1}{3} \big(1 - 3^{-s} \big) \sum_{\chi \in \widehat{G_{(3 \sqrt{-3})}}} F((3\sqrt{-3}), \chi, s)) \Bigg].
\end{multline}

The first three sums are over fields of discriminant $D^* = -D/3$, and the 
last sum also includes fields of discriminant $-27D$. 
For the same reasons as discussed for $p \neq 3$ we have 
$\om_{\chi}(3)=\om_{E}(3)$, where $E$ is the cubic field associated to $\chi$.

We thus obtain a contribution of 
$1 + 2 \cdot 3^{-s} + 3 \om_{E}(3) \cdot 3^{-2s}$
for each field of discriminant $D^*$, and of $1 - 3^{-s}$ for each field of 
discriminant $-27D$, in accordance with the theorem.

\section{Numerical Examples}\label{sec_examples}

We present some numerical examples of our main results.

Suppose first that $D < 0$.

If $(\rk_3(D),\rk_3(D^*))=(0,0)$, then there are no cubic fields of
discriminant $D^*$ or $-27D$.
\begin{equation}
\Phi_{-4}(s)=\dfrac{1}{2}\left(1+\dfrac{2}{3^{2s}}\right)\prod_{\leg{12}{p}=1}\left(1+\dfrac{2}{p^s}\right).
\end{equation}

If $(\rk_3(D),\rk_3(D^*))=(1,0)$, there are no cubic fields of discriminant 
$D^*$ and a unique cubic field of discriminant $-27D$.

\begin{equation}
\Phi_{-255}(s)
=\dfrac{1}{2}\left(1+\dfrac{2}{3^s}+\dfrac{6}{3^{2s}}\right)\prod_{\leg{6885}{p}=1}\left(1+\dfrac{2}{p^s}\right)\\
+\left(1-\dfrac{1}{3^s}\right)\prod_p\left(1+\dfrac{\om_{L6885}(p)}{p^s}\right)\;,\end{equation}
where $L6885$ is the cubic field determined by $x^3 - 12x - 1 = 0$.

If $(\rk_3(D),\rk_3(D^*))=(1,1)$, there is a unique cubic field of 
discriminant $D^*$ and no cubic fields of discriminant $-27D$.

\begin{equation}
\Phi_{-107}(s)=\dfrac{1}{2}\left(1+\dfrac{2}{3^{2s}}\right)\prod_{\leg{321}{p}=1}\left(1+\dfrac{2}{p^s}\right)\\
+\left(1+\dfrac{2}{3^{2s}}\right)\prod_p\left(1+\dfrac{\om_{L321}(p)}{p^s}\right),
\end{equation}
where $L321$ is the field determined  by $x^3-x^2-4x+1$.

If $(\rk_3(D),\rk_3(D^*))=(2,1)$, there is a unique cubic field of 
discriminant $D^*$ and three cubic fields of discriminant $-27D$.

\begin{align*}
\Phi_{-8751}(s)=&\dfrac{1}{2}\left(1+\dfrac{2}{3^s}+\dfrac{6}{3^{2s}}\right)\prod_{\leg{26253}{p}=1}\left(1+\dfrac{2}{p^s}\right)\\
&
+\left(1+\dfrac{2}{3^s}-\dfrac{3}{3^{2s}}\right)\prod_{p\ne3}\left(1+\dfrac{\om_{L2917}(p)}{p^s}\right)
+\left(1-\dfrac{1}{3^s}\right)\sum_{1\le i\le 3}\prod_p\left(1+\dfrac{\om_{L236277_i}(p)}{p^s}\right),
\end{align*}
where the four fields above are defined as follows:

\bigskip
\centerline{
\begin{tabular}{|c||c|c|}
\hline
Cubic field & Discriminant & Defining polynomial\\
\hline\hline
\hline
$L2917$&$8751/3$&$x^3-x^2-13x+20$\\
\hline
$L236277_1$&$3^3\cdot8751$&$x^3-138x+413$\\
\hline
$L236277_2$&$3^3\cdot8751$&$x^3-129x-532$\\
\hline
$L236277_3$&$3^3\cdot8751$&$x^3-90x-171$\\
\hline
\end{tabular}}
\bigskip
If $(\rk_3(D),\rk_3(D^*))=(2,2)$,
there are four cubic fields of discriminant $D^*$ and none of
discriminant $-27D$. Recall from Proposition \ref{case22} above that
if $D\equiv6\pmod9$, $3$ is totally split in one of them and inert in the
other three, so  one of the cubic fields of discriminant $D^*$, which we 
include first, is distinguished by the fact that $3$ is totally split.

\begin{equation}
\Phi_{-34603}(s) = \dfrac{1}{2}\left(1+\dfrac{2}{3^{2s}}\right)\prod_{\leg{103809}{p}=1}\left(1+\dfrac{2}{p^s}\right)
+\left(1+\dfrac{2}{3^{2s}}\right)\sum_{1\le i\le 4}\prod_p\left(1+\dfrac{\om_{L103809_i}(p)}{p^s}\right);
\end{equation}
\centerline{
\begin{tabular}{|c||c|c|}
\hline
Cubic field & Discriminant & Defining polynomial\\
\hline\hline
$L103809_1$&$3\cdot34603$&$x^3-x^2-84x+261$\\
\hline
$L103809_2$&$3\cdot34603$&$x^3-x^2-64x+91$\\
\hline
$L103809_3$&$3\cdot34603$&$x^3-x^2-92x-204$\\
\hline
$L103809_4$&$3\cdot34603$&$x^3-x^2-62x-15$\\
\hline
\end{tabular}}
\bigskip

The case $D > 1$ is very similar, so we will only give one example. If
(for example) $(\rk_3(D),\rk_3(D^*))=(1,1)$ there is a unique cubic field of 
discriminant $D^*$ and three cubic fields of discriminant $-27D$.

\begin{align*}
3\Phi_{321}(s)&=\dfrac{1}{2}\left(1+\dfrac{2}{3^s}+\dfrac{6}{3^{2s}}\right)\prod_{\leg{-963}{p}=1}\left(1+\dfrac{2}{p^s}\right)\\
&\phantom{=}+\left(1+\dfrac{2}{3^s}-\dfrac{3}{3^{2s}}\right)\prod_{p\ne3}\left(1+\dfrac{\om_{LM107}(p)}{p^s}\right)
+\left(1-\dfrac{1}{3^s}\right)\sum_{1\le i\le 3}\prod_p\left(1+\dfrac{\om_{LM8667_i}}{p^s}\right),
\end{align*}
where the indicated cubic fields are given as follows:
\bigskip

\centerline{
\begin{tabular}{|c||c|c|}
\hline
Cubic field & Discriminant & Defining polynomial\\
\hline\hline
$LM107$&$-321/3$&$x^3-x^2+3x-2$\\
\hline
$LM8667_1$&$-3^3\cdot321$&$x^3+18x-45$\\
\hline
$LM8667_2$&$-3^3\cdot321$&$x^3+6x-17$\\
\hline
$LM8667_3$&$-3^3\cdot321$&$x^3+15x-28$\\
\hline
\end{tabular}}
\bigskip

\section{Counting $S_3$-sextic fields of bounded discriminant\protect\footnote{We thank the anonymous referee of \cite{TT6} for
(somewhat indirectly) suggesting this application.}}\label{sec_computations}

Theorem \ref{thm_main_cubic} naturally lends itself to counting $S_3$-sextic fields, that is, fields which are
Galois over $\Q$ with Galois group $S_3$.
For any such $\wK$ with cubic and quadratic subfields $K$ and $k$ respectively, we have the formula 
\begin{equation}
\Disc(\wK) = \Disc(K)^2 \Disc(k) = \Disc(k)^3 f(K)^4.
\end{equation}
Let $N^{\pm}(X; S_3)$ denote the number of $S_3$-sextic fields $\wK$ with $0 < \pm \Disc(\wK) < X$. Then Theorem \ref{thm_main_cubic}
may be used to compute $N^{\pm}(X; S_3)$: iterate over fundamental discriminants $D$
with $0 < \pm D < X^{1/3}$; compute the Dirichlet series $\Phi_D(s)$ to $f(K) < (X/D^3)^{1/4}$ and evaluate its partial sums; finally,
sum the results.

We implemented this algorithm in PARI/GP \cite{pari}, which can easily handle
the various quantities occurring in \eqref{eqn_main_cubic}. For a list of cubic fields we relied on Belabas's \url{cubic} program \cite{Bel}.

We used the GP calculator, which has the advantage of simplicity. The 
disadvantage of this approach is that we were 
obliged to read the output of \url{cubic} from disk, limiting us by available disk space. 
One could probably
compute $N^{\pm}(X; S_3)$ to at least $X = 10^{27}$ by directly implementing \eqref{eqn_main_cubic} within Belabas's code; alternatively,
Belabas has informed us that an implementation of \url{cubic} within PARI/GP may be forthcoming. In any case we leave further
computations for later.

This approach dictated a slight variant of the algorithm described above:
\begin{itemize}
\item We parsed Belabas's output into files readable by PARI/GP, using a Java program written for this purpose.
\item Given a table of all cubic fields $K$ with $0 < \pm \Disc(K) < Y$ for some $Y$, it must contain all fields in
$\LL_3(D)$ with $0 < \mp \Disc(K) < Y/27$, allowing us to choose any $X \leq 3^{-9} Y^3$.
\item Processing each cubic field in turn, and ignoring those not in $\calL_3(D)$ for some fundamental discriminant $D$
with $|D| \leq X^{1/3}$, we computed the associated Dirichlet series to a length of $\lfloor (X/D^3)^{1/4} \rfloor$, and its partial sum
(less the $\frac{1}{2}$ term for $f(K) = 1$), and maintained a running total of the results.
\item Finally, we added the main term of \eqref{eqn_main_cubic} for each $D$ with $0 < \pm D < X^{1/3}$.
\end{itemize}
Our algorithm would also allow for efficient computation of the $\Phi_D(s)$, given a {\itshape sorted} version of Belabas's output.

The implementation posed no particular difficulties, and our PARI/GP 
source code is available on the second author's website.\footnote{To replicate
our data for large $X$ one must also install and run Belabas's \url{cubic} program, and as well as our parser.
With our source code we have also made available a modestly sized table
of cubic fields, with which our PARI/GP program suffices alone to replicate our data for smaller values of $X$.} On a 2.1 GHz MacBook 
our computation
took approximately 3 hours for negative discriminants $< 3 \cdot 10^{23}$ and 10 hours for positive discriminants $< 10^{23}$;
it is to be expected from the shape of \eqref{eqn_main_cubic} that counts of negative discriminants may be computed more efficiently
than positive, even though there are more of them.

This brings
us to our data:
\\
\\
\begin{center}
\begin{tabular}{c | c | cccc}
$X$ & $N_6^+(X; S_3)$ \\ \hline
$10^{12}$ & 690\\
$10^{13}$ & 1650 \\ 
$10^{14}$& 3848 \\
$10^{15}$& 8867 \\
$10^{16}$& 20062 \\
$10^{17}$& 45054 \\
$10^{18}$& 100335 \\
$10^{19}$& 222939 \\
$10^{20}$& 492335 \\
$10^{21}$& 1083761 \\
$10^{22}$& 2378358 \\
$10^{23}$& 5207310 \\
- & - \\

\end{tabular}
\ \ \ \ \ \ \ \
\begin{tabular}{c | c | ccc}
$X$ &$N_6^-(X; S_3)$\\ \hline
$10^{12}$ & 2809\\
$10^{13}$ & 6315\\
$10^{14}$& 14121\\
$10^{15}$& 31276\\
$10^{16}$& 68972\\
$10^{17}$& 151877\\
$10^{18}$& 333398\\
$10^{19}$& 729572\\
$10^{20}$& 1592941\\
$10^{21}$& 3470007\\
$10^{22}$& 7550171\\
$10^{23}$& 16399890\\
$3 \cdot 10^{23}$& 23738460\\

\end{tabular}
\end{center}
\vskip 0.2in
This data may be compared to known theoretical results on $N^{\pm}(X; S_3)$.
It was proved by
Belabas-Fouvry \cite{BF} and Bhargava-Wood \cite{BW} (independently) $N^{\pm}(X; S_3) \sim B^{\pm} X^{1/3}$
for explicit constants $B^{\pm}$ (with $B^- = 3 B^+$), and in \cite{TT6} Taniguchi and the second author obtained a
power saving error term.

The authors of \cite{TT6} also computed tables of $N^{\pm}(X; S_3)$ up to $X = 10^{18}$ using a different method, allowing us
to double-check our work here.
Based on this data and on \cite{R, BST, TT}, 
they guessed the existence of a secondary term of order $X^{5/18}$, and found that the data further suggested the existence
of additional, unexplained lower order terms. For more on this we
defer to \cite{TT6}.

\end{document}